\documentclass[12pt,a4paper,reqno]{amsart}
\usepackage[hmargin=2.5cm,bmargin=2.5cm,tmargin=3cm]{geometry}
\usepackage{enumerate}
\usepackage{amsmath,amsthm,amssymb,amsfonts,latexsym}
\usepackage{orcidlink}
\usepackage[gen]{eurosym}
\usepackage[german,english]{babel}

\usepackage{latexsym}

\usepackage{tikz}					

\newcommand{\e}{\mathrm e}

\newcommand{\NN}{\mathbb{N}}

\newcommand{\QDir}{Q^{(D)}}
\newcommand{\QNeu}{Q^{(N)}}

\newtheorem{theorem}{Theorem}
\newtheorem{lemma}[theorem]{Lemma}

\newtheorem{proposition}[theorem]{Proposition}

\newtheorem{remark}[theorem]{Remark}

\newtheorem{example}[theorem]{Example}

\newcommand{\be}{\begin{equation}}
\newcommand{\ee}{\end{equation}}
\newcommand{\bea}{\begin{eqnarray*}}
	\newcommand{\eea}{\end{eqnarray*}}
\newcommand{\beq}{\begin{eqnarray}}
\newcommand{\eeq}{\end{eqnarray}}

\newtheorem{cor}[theorem]{Corollary}
\newtheorem{ass}[theorem]{Assumption}

\usepackage{todonotes}

\title[Spectral comparison results for Laplacians on discrete graphs]{Spectral comparison results for Laplacians on discrete graphs} 

\subjclass[2010]{}


\author[P.~Bifulco]{Patrizio Bifulco\orcidlink{0009-0004-0628-374X}}
\author[J.~Kerner]{Joachim Kerner\orcidlink{0000-0003-0638-4183}}
\author[C.~Rose]{Christian Rose\orcidlink{0000-0002-1078-0798}}

\address{Patrizio Bifulco, Lehrgebiet Analysis, Fakult\"at Mathematik und Informatik, Fern\-Universit\"at in Hagen, D-58084 Hagen, Germany}
\email{patrizio.bifulco@fernuni-hagen.de}

\address{Joachim Kerner, Lehrgebiet Angewandte Stochastik, Fakult\"at Mathematik und Informatik, Fern\-Universit\"at in Hagen, D-58084 Hagen, Germany}
\email{joachim.kerner@fernuni-hagen.de}

\address{Christian Rose, Institut für Mathematik, Universit\"at Potsdam, D-14476 Potsdam, Germany}
\email{christian.rose@uni-potsdam.de}

\date{\today}

\thanks{
}

\begin{document}

	\begin{abstract} In the recent literature, various authors have studied spectral comparison results for Schr\"odinger operators with discrete spectrum in different settings including Euclidean domains and quantum graphs. In this note we derive such spectral comparison results in a rather general framework for general and possibly infinite discrete graphs. Along the way, we establish a discrete version of the local Weyl law whose proof does neither involve any Tauberian theorem nor the Weyl law as used in the continuous case.
	\end{abstract}
	
\maketitle

\section{Introduction}\label{sec:introduction}
The underlying idea of the spectral comparison results we shall study in this paper is to compare -- in a suitable way -- the eigenvalues of two different (self-adjoint) operators defined over the same structure, and where one of them is a perturbation of the other. 

A benchmark in this direction was obtained in \cite{RWY} (see also \cite{FrankLarson} for more recent developments) where the authors compare the sums of the first $n$ eigenvalues of the Neumann Laplacian and of a Robin Laplacian on a bounded domain in $\mathbb{R}^2$. Based on this, they derive an explicit expression for the limiting mean eigenvalue distance which involves the circumference, the area of the domain and the Robin parameter. Similar results have subsequently been established for quantum graphs~\cite{BKDeltaPrime,BK23,BKInfinite,BandSchanzSofer}. More explicitly, given a metric graph, the authors of \cite{BK23} compare the sums of the first $n$ eigenvalues of two self-adjoint Schr\"odinger operators with each other. For finite compact metric graphs, an explicit expression for the mean eigenvalue distance involves some combinatorial data of the graph as well as the potentials associated with the operators. In contrast to this, in \cite{BKInfinite} it is shown that, for a certain infinite quantum graph of infinite total length, the mean average eigenvalue distance has value zero, but that a modified average of the eigenvalue distances yields a more meaningful result. This remarkable and somewhat unexpected effect has its root in a modified Weyl law for the eigenvalue counting function. As shown in \cite{BKInfinite}, this modified Weyl law also leads to a modified \textit{local} Weyl law. Indeed, local Weyl laws constitute a key tool for the derivation of spectral comparison results and they are intimately connected to properties of the heat kernel. As we will see in the following, for discrete graphs, the derivation of a local Weyl law is more direct than in other settings in that it does not involve Weyl's law. In contrast, in many settings including Euclidean domains and metric graphs, there is a direct link between the heat kernel and the Weyl law via Karamata's Tauberian theorem.

In this paper, we obtain spectral comparison results for Laplacians on discrete graphs which are typically infinite (Theorem~\ref{thm:spectral-comparison}) and, to this end, we provide a local Weyl law for such graphs (Proposition~\ref{prop:local-weyl-law}). By doing this, we also recover some results obtained in \cite{BKDiscrete} for normalized Schr\"odinger operators on finite discrete graphs. The surprising feature of the results obtained in this paper is that one has yet again to modify the spectral comparison results as well as the local Weyl law. Indeed, rather than an expression for an \textit{averaged version} of the eigenvalues distances one obtains an expression for the sum of all eigenvalue differences. In this sense, the spectral properties of Laplacians on discrete graphs are quite different from those of Schr\"odinger operators defined in the continuous setting. In addition, the derived spectral comparison results immediately imply an Ambarzumian-type theorem (Corollary~\ref{AmbarzumianGraphs}) and hence a connection to inverse spectral theory is established; we refer to \cite{Ambarz,Borg,Davies:2013,KurasovBook,BKAmba} and references therein.
\section{Basic setup and the main result}\label{sec:setup}
Let $X$ be a non-empty and at most countable set and $b: X \times X \rightarrow [0,\infty)$ symmetric satisfying
$b(x,x) = 0$ 
and $\sum_{y \in X} b(x,y) < \infty$ for all $x \in X$.
Let $c: X \rightarrow [0,\infty)$ and $m:X \rightarrow (0,\infty)$ be two maps and extend $m$ to a measure in the obvious way. We then call $(b,c)$ a \emph{graph} over the discrete measure space $(X,m)$; see \cite{KellerLenzGraphs} for more details.

Denote by $C(X)$ the linear space of real functions and 
\[
\ell^2(X,m) := \bigg\{ f \in C(X) \: : \: \sum_{x \in X}  |f(x)|^2 m(x) < \infty \bigg\}
\]
which, equipped with the scalar product $$\langle f,g\rangle_{\ell^2(X,m)} := \sum_{x \in X} f(x)g(x) m(x)\ , \qquad f,g \in \ell^2(X,m)\ ,$$ is a Hilbert space. We  
let
\[
    \mathcal{D}_{b,c} := \Bigg\{ f \in C(X) \: : \: \sum_{x,y \in X} b(x,y)|f(x)-f(y)|^2 + \sum_{x \in X} c(x) |f(x)|^2 < \infty \Bigg\},
\]
and define the quadratic form $\mathcal{Q}_{b,c}: \mathcal{D}_{b,c} \times \mathcal{D}_{b,c} \rightarrow \mathbb{R}$ by
\[
\mathcal{Q}_{b,c}(f,g) := \frac{1}{2}\sum_{x,y \in X} b(x,y) (f(x)-f(y)) (g(x)-g(y)) + \sum_{x \in X}c(x) f(x)g(x)\ , \quad f,g \in \mathcal{D}_{b,c}\ .
\]
We are interested in self-adjoint realizations of the operator $ \mathcal{L}_{b,c}$ formally acting as
\[
\big(\mathcal{L}_{b,c}f \big)(x) := \frac{1}{m(x)} \sum_{y \in X} b(x,y) (f(x)-f(y)) + \frac{c(x)}{m(x)} f(x)\ , \qquad  \text{$x \in X$}\ .
\]
To this end, let $Q=\mathcal{Q}_{b,c}$ be a positive quadratic form over $(b,c)$ with form domain $D(Q) \subset \ell^2(X,m)$ 
which is closed with respect to the form norm
\[
\Vert f \Vert_{Q}^2 := \Vert f \Vert_{\ell^2(X,m)}^2 + Q(f,f)\ , \qquad f \in D(Q)\ .
\]
We assume that $C_c(X) \subset D(Q)$, where $C_c(X)$ denotes the subspace of functions in  $C(X)$ having {finite support}.

According to the form representation theorem, its associated self-adjoint positive operator $L$ in $\ell^2(X,m)$ with domain $D(L) \subset D(Q)$ satisfies the relation
\[
Q(f,g) =\langle f, L g \rangle_{\ell^2(X,m)}  \qquad \text{for $g \in D(L)$, $f \in D(Q)$}\ .
\]
As demonstrated in \cite{KellerLenzGraphs}, we have $L=\mathcal{L}_{b,c}$ on $D(L)$. Most prominent choices are 
\begin{align*}
\QNeu(f,g) &:= \mathcal{Q}_{b,c}(f,g) \qquad \text{for}\  f,g \in D(\QNeu):= \ell^2(X,m) \cap \mathcal{D}_{b,c}\ ,
\end{align*}
and 
\begin{align*}
\hspace*{-.9cm}\QDir(f,g) &:= \mathcal{Q}_{b,c}(f,g) \qquad \text{for}\  f,g \in D(\QDir):=
\overline{C_c(X)}^{\Vert \cdot \Vert_{\mathcal{Q}_{b,c}}}\ ,
\end{align*}
which correspond to the Neumann and Dirichlet realizations, respectively. By construction, we always have $D(\QDir) \subset D(Q) \subset D(\QNeu)$. Such inclusions hold, in particular, for so-called \emph{Markovian realizations}, see \cite[Definition~3.9 and Theorem~3.11]{KellerLenzGraphs}.  
\begin{remark} As illustrated in \cite{KellerLenzGraphs}, there exist graphs satisfying $D(\QDir)=D(\QNeu)$ and therefore there is a unique form in our sense associated with such a graph. However, there also exist many situations where the inclusion is indeed strict and in which case the corresponding realizations in-between can be characterized {explicitly}, cf.\ \cite{KellerLenzSchmidtSchwarz2019}.
\end{remark}
Regarding spectral comparison results, we now want to compare -- for a fixed discrete measure space $(X,m)$ -- the spectrum of two self-adjoint operators, one defined over the graph $(b,0)$ and the other over $(b,c)$ for some $c:X \rightarrow [0,\infty)$. More explicitly, we want to compare the spectrum of $(L_0,D(L_0))$ associated with the form $(Q_0,D(Q_0))$ defined over a graph $(b,0)$ over $(X,m)$ to the one of $(L_c,D(L_c))$ associated with the form $(Q_c,D(Q_c))$ over the graph $(b,c)$ over $(X,m)$, where $Q_c=\mathcal{Q}_{b,c}$ and $$D(Q_c)=D(Q_0)\cap \bigg\{f\in \ell^2(X,m)\colon \sum_{x\in X} c(x)|f(x)|^2<\infty\bigg\} \subset D(Q_0).$$
In the following, we call such a self-adjoint operator $L_0$ a \textit{realization} and the self-adjoint operator $L_c$ corresponding to $c:X \rightarrow [0,\infty)$ its \textit{induced} realization.
\begin{ass}[Discreteness of the spectrum]\label{ass:discrete-spec}
   Let $L_0$ be a realization associated with $(Q_0,D(Q_0))$ defined over a graph $(b,0)$ over the discrete measure space $(X,m)$. We assume that its form domain $D(Q_0)$ is compactly embedded in $\ell^2(X,m)$. 
\end{ass}
Following a general argument~cf.~\cite{schmudgen2012unbounded}, the spectrum of $L_0$ is purely discrete under Assumption~\ref{ass:discrete-spec}; in this case we write $\lambda_n(0)$ for the $n$-th eigenvalue of $L_0$. Notice that, whenever $\ell^2(X,m)$ is infinite-dimensional, the eigenvalues diverge to infinity and hence $L_0$ is an unbounded operator. Also, given $X$ is finite, Assumption~\ref{ass:discrete-spec} is superfluous. On the other hand, since $D(Q_0) \subset D(\QNeu_0):= \ell^2(X,m) \cap \mathcal{D}_{b,0}$, one has a compact embedding whenever $D(\QNeu_0)$ is compactly embedded in $\ell^2(X,m)$. This holds, for instance, whenever $m(X) < \infty$ and $X$ satisfies a Sobolev inequality \cite{HuaKSW-23} or is bounded with respect to a certain distance-like function, \cite[Lemma~2.6]{BifMug25}.
\begin{lemma}\label{LemmaDiscreteLC} Consider a realization $L_0$, $c:X \rightarrow [0,\infty)$ and its induced realization $L_c$. If $L_0$ satisfies Assumption~\ref{ass:discrete-spec}, $D(Q_c)$ is also compactly embedded in $\ell^2(X,m)$ and hence $\sigma(L_c)$ is also purely discrete. 
\end{lemma}
\begin{proof} The statement follows directly from the definition of a compact embedding, taking into account the definition of $D(Q_c)$ and $c\geq 0$ on $X$.
\end{proof}
\begin{example}\label{ExampleI} Let $X=\mathbb{N}$ with $m(n):=n^{-4}$ be given. We consider the path graph over $(X,m)$ defined via $b(n+1,n)=b(n,n+1)=n^2$ and $b(n,m)=0$ whenever $|n-m| > 1$ (cf.\ Figure \ref{fig:path-graph} below). Moreover, we suppose that $c(n)=0$ for all $n \in \mathbb{N}$. Then, employing~\cite[Lemma~2.6]{BifMug25}, one has that $D(Q^{(N)})$ is compactly embedded in $\ell^2(X,m)$ and hence any self-adjoint realization $L_0$ has purely discrete spectrum.
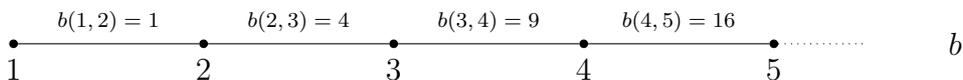
\begin{figure}[h]
\begin{tikzpicture}[scale=0.50]
      \tikzset{enclosed/.style={draw, circle, inner sep=0pt, minimum size=.10cm, fill=black}}

      \node[enclosed, label={below: $3$}] (C) at (10,4) {};
      \node[enclosed, label={below: $1$}] (A) at (0,4) {};
      \node[enclosed, label={below: $2$}] (B) at (5,4) {};
      \node[enclosed, label={below: $4$}] (D) at (15,4) {};
      \node[enclosed, label={below: $5$}] (E) at (20,4) {};
      \node[enclosed, label = {right: \qquad $b$}, white] (F) at (22.5,4) {};

      \draw (A) -- (B) node[midway, above] (edge1) {\tiny $b(1,2)=1$};
      \draw (B) -- (C) node[midway, above] (edge2) {\tiny $b(2,3)=4$};
      \draw (C) -- (D) node[midway, above] (edge3) {\tiny $b(3,4)=9$};
      \draw (D) -- (E) node[midway, above] (edge4) {\tiny $b(4,5)=16$};
      \draw (E) edge[dotted] (F) node[midway, above] (edge4) {};
     \end{tikzpicture}
     \vspace{-1.5cm}
     \caption{The infinite path graph $b$ over $(\mathbb{N},m)$.}\label{fig:path-graph}
     \end{figure}
\end{example}
Whenever an induced realization $L_c$ has purely discrete spectrum, we likewise denote its eigenvalues by $\lambda_n(c)$, $n \in \mathbb{N}$. We can now formulate the main result of this note.
\begin{theorem}[Spectral comparison for discrete graphs]\label{thm:spectral-comparison}
    Let $L_0$ satisfy Assumption \ref{ass:discrete-spec}, $c\colon X\to[0,\infty)$ and $L_c$ the corresponding induced realization. 
    Then, we have
    \begin{align*}
        \sum_{n=1}^{\dim \ell^2(X,m)} \big( \lambda_n(c) - \lambda_n(0) \big) = \sum_{x \in X} \frac{c(x)}{m(x)}\ ,
    \end{align*}
    where the right-hand side equals $+\infty$ if the map $X \ni x \mapsto \frac{c(x)}{m(x)}$ does not belong to $\ell^1(X, \mathbf{1})$. Note that both sums in the above formula are finite iff the underlying set $X$ is finite. 
\end{theorem}
\begin{remark} Let $X$ be infinite. Then, under Assumption~\ref{ass:discrete-spec} and $\frac{c}{m} \in \ell^1(X,\mathbf{1})$, Theorem~\ref{thm:spectral-comparison} implies that $\left(\lambda_n(c) - \lambda_n(0)\right)_{n\in \mathbb{N}}$ forms a null sequence. This therefore shows that the two (unbounded) operators $L_0$ and $L_c$ are asymptotically isospectral; see~\cite{KurasovSuhrAsymp} for a related discussion.
\end{remark}
\begin{example} Consider Example~\ref{ExampleI} where $X=\mathbb{N}$ and let $c:X\rightarrow [0,\infty)$ be defined via $c(n):=n^2$, $n \in \mathbb{N}$.
Then, applying Theorem~\ref{thm:spectral-comparison}, we obtain
\begin{equation*}
  \sum_{n=1}^\infty \big( \lambda_n(c) - \lambda_n(0) \big) = \sum_{n \in \mathbb{N}} \frac{1}{n^2}=\frac{\pi^2}{6}\ .
\end{equation*}
\end{example}
\section{Proof of the main result}\label{sec:proof}

It is well-known that any induced realization $L_c$ 
generates a strongly continuous semigroup $(\e^{-tL_c})_{t \geq 0}$ consisting of integral operators on $\ell^2(X,m)$, i.e., there exists a map $p^c : [0,\infty) \times X \times X \rightarrow \mathbb{R}$, the \emph{heat kernel}, such that
\[
\big( \e^{-tL_c} f \big)(x) = \sum_{y \in X} p^c_t(x,y)f(y)m(y), \qquad x\in X, \ t>0, \ f\in\ell^2(X,m).
\]
Furthermore, this identity readily implies 
\[
p^c_t(x,y) = \frac{\langle \mathbf{1}_{\{x\}}, \e^{-tL_c} \mathbf{1}_{\{y\}}\rangle_{\ell^2(X,m)}}{m(x)m(y)}, \qquad x,y\in X, \ t>0.
\]
This yields the following statement which represents a counterpart to \cite[Theorem~4.1]{BHJ}, \cite[Proposition~4]{BK23} and \cite[Lemma~13]{BKInfinite} for discrete graphs.
\begin{proposition}[Local Weyl law]\label{prop:local-weyl-law}
    Let $L_0$ satisfy Assumption \ref{ass:discrete-spec}, $c\colon X\to[0,\infty)$ and
    $L_c$ be the induced realization. Further, let $f_n^c \in \ell^2(X,m)$, $n=1,\dots,\dim \ell^2(X,m)$, denote 
    corresponding orthogonal and normalized eigenfunctions to the eigenvalues $\lambda_n(c)$. Then, one has
    \begin{align*}
        \sum_{n=1}^{\dim \ell^2(X,m)} \vert f_n^c(x) \vert^2  = \frac{1}{m(x)}, \qquad x \in X.
    \end{align*}
\end{proposition}
\begin{proof}
As $\sigma(L_c)$ is purely discrete by Lemma~\ref{LemmaDiscreteLC} and $L_c$ is associated with a positive form, we employ Mercer's theorem to obtain
\begin{align*}
    p^c_t(x,x) = \sum_{n=1}^{\dim \ell^2(X,m)} \e^{-t\lambda_n} \vert f_n^c(x) \vert^2
\end{align*}
for all $t>0$ and $x \in X$, where the right-hand side converges uniformly in $x \in X$ for fixed $t>0$. 
Due to continuity of scalar products and the strong continuity of the semigroup $(\e^{-tL_c})_{t \geq 0}$, we have 
\begin{align*}
    \lim_{t \rightarrow 0^+} p^c_t(x,x) = \frac{\Big\langle \mathbf{1}_{\{x\}}, \lim\limits_{t \rightarrow 0^+} \e^{-tL_c}\mathbf{1}_{\{x\}} \Big\rangle_{\ell^2(X,m)}}{m(x)^2} = \frac{\Vert \mathbf{1}_{\{x\}} \Vert^2_{\ell^2(X,m)}}{m(x)^2} = \frac{1}{m(x)}
\end{align*}
for every $x \in X$.
 Moreover, using Mercer's identity and Fatou's lemma, we conclude for every $x \in X$
 and every $s > 0$ 
\begin{align*}
    p^c_s(x,x) &\leq 
    \sum_{n=1}^{\dim \ell^2(X,m)} \lim_{t \rightarrow 0^+} \e^{-t\lambda_n(c)}\vert f_n^c(x) \vert^2  \leq  \liminf_{t \rightarrow 0^+} \sum_{n=1}^{\dim \ell^2(X,m)} \e^{-t\lambda_n(c)}\vert f_n^c(x) \vert^2 \\ &= \lim_{t \rightarrow 0^+} p_t^c(x,x)= \frac{1}{m(x)}\ .
\end{align*}
The claim follows considering $s \rightarrow 0^+$.
\end{proof}
	In order to relate the eigenvalues of $L_c$ and $L_0$, we want to ensure that the respective form domains are the same. 
    \begin{lemma}
         Let $L_0$ be a realization over the graph $(b,0)$, $c\colon X\to[0,\infty)$ and
     $L_c$ the induced realization over the graph $(b,c)$. Moreover, suppose $\frac{c}{m}\in\ell^\infty(X,\mathbf{1})$. Then, we have 
        \[
        D(Q_0) = D(Q_c)\ .
        \]
    \end{lemma}
	\begin{proof}
	    By definition of $Q_c$, the inclusion $D(Q_0) \supset D(Q_c)$ is clear. If conversely $f \in D(Q_0)$
        it is sufficient to show that $\sum_{x \in X} c(x) |f(x)|^2 < \infty$ to conclude that $f \in D(Q_c)$. As $\frac{c}{m} \in \ell^\infty(X)$ and $f \in \ell^2(X,m)$, we deduce
        \[
        \sum_{x \in X} c(x) |f(x)|^2 = \sum_{x \in X} \frac{c(x)}{m(x)}  |f(x)|^2m(x) \leq \left\Vert \frac{c}{m} \right\Vert_{\ell^\infty(X)} \Vert f \Vert_{\ell^2(X,m)}^2 < \infty\ ,
        \]
        immediately yielding the claim.
	\end{proof}
    Whenever the form domains agree, similar arguments as in \cite{LSHadamard} yield the following; see also~\cite{Kat66}.
    \begin{proposition}[Hadamard-type formula]\label{prop:hadamard}
        Let $L_0$ satisfy Assumption \ref{ass:discrete-spec}, $c\colon X\to[0,\infty)$ and for any $\tau\in[0,1]$ let $L_{\tau c}$ be the induced realization of the function $\tau c$. Further, for any $\tau\in[0,1]$, let $f_n^{\tau c} \in \ell^2(X,m)$, $n=1,\dots,\dim \ell^2(X,m)$, denote 
    the 
    corresponding orthogonal and normalized eigenfunctions to the eigenvalues $\lambda_n(\tau c)$. Assume $\frac{c}{m} \in \ell^\infty(X,\mathbf{1})$. Then, the map $[0,1] \ni \tau \mapsto \lambda_n(\tau c)$ is differentiable almost everywhere with
        \begin{align*}
            \frac{\mathrm{d}}{\mathrm{d}\tau} \lambda_n(\tau c) = \sum_{x \in X} c(x) \vert f_n^{\tau c}(x) \vert^2\ .
        \end{align*}
    \end{proposition}
    Using this, we are now in the position to prove our main result.
    \begin{proof}[Proof of Theorem \ref{thm:spectral-comparison}]
        We prove the statement for the more involved infinite case where $\dim \ell^2(X,m)=\infty$; the finite case is then straightforward. We first restrict to the case where $\frac{c}{m} \in \ell^1(X,\mathbf{1})$: 
        by the Hadamard-type formula, we may write for any $N\in\NN$
        \begin{align*}
            \sum_{n=1}^N \big( \lambda_n(c) - \lambda_n(0) \big) = \int_0^1 \sum_{x \in X} c(x) \sum_{n=1}^N \vert f_n^{\tau c}(x) \vert^2 \ \mathrm{d}\tau\ .
        \end{align*}
       For every $N \in \mathbb{N}$ and every $\tau \in [0,1]$, the local Weyl law provided by Proposition \ref{prop:local-weyl-law} yields  $\sum_{n=1}^N \vert f_n^{\tau c}(x) \vert^2 \leq \sum_{n=1}^\infty \vert f_n^{\tau c}(x) \vert^2 = \frac{1}{m(x)}$. 
       Hence, we get 
        \begin{align*}
            \sum_{x \in X} c(x) \sum_{n=1}^N \vert f_n^{\tau c}(x) \vert^2 \leq \left\Vert \frac{c}{m} \right\Vert_{\ell^1(X)} < \infty\ ,
        \end{align*}
        such that the sequence $(\sum_{x \in X} c(x) \sum_{n=1}^N \vert f_n^{\tau c}(x) \vert^2)_{N \in \mathbb{N}}$ is bounded uniformly in $\tau \in [0,1]$. Therefore, due to dominated convergence, we arrive at
        \begin{align*}
            \sum_{n=1}^\infty \big( \lambda_n(c) - \lambda_n(0) \big) &= \int_0^1 \sum_{x \in X} c(x) \lim_{N \rightarrow \infty} \bigg(\sum_{n=1}^N \vert f_n^{\tau c}(x) \vert^2 \ \bigg)\mathrm{d}\tau \\&= \int_0^1 \sum_{x \in X} \frac{c(x)}{m(x)} \ \mathrm{d}\tau = \sum_{x \in X} \frac{c(x)}{m(x)}\ .
        \end{align*}

        Next, assume $\frac{c}{m} \notin \ell^1(X,\mathbf{1})$: similar to \cite[Theorem~9]{BKInfinite} we enumerate the set of vertices $X = \{ x_n \}$ and  consider the family of maps $c_M: X \rightarrow [0,\infty)$, $M \in \mathbb{N}$, given by
        \begin{align*}
            c_M(x_n) := \begin{cases} c(x_n), & \text{if $1 \leq n \leq M$}, \\ 0, & \text{else.}
            \end{cases}
        \end{align*}
        Then, clearly $\frac{c_M}{m} \in \ell^1(X,\mathbf{1})$ for every $M \in \mathbb{N}$ and according to the min-max principle and the first part of the proof, we deduce
        \begin{align*}
            \sum_{n=1}^\infty \big( \lambda_n(c) - \lambda_n(0) \big) \geq \sum_{n=1}^\infty \big( \lambda_n(c_M) - \lambda_n(0) \big) = \sum_{x \in X} \frac{c_M(x)}{m(x)} = \sum_{n=1}^M \frac{c(x_n)}{m(x_n)}
        \end{align*}
        for every $M \in \mathbb{N}$. Letting $M\to\infty$, the statement follows. 
    \end{proof}
    We finish this note with an interesting and straightforward application of Theorem~\ref{thm:spectral-comparison}, namely, we shall prove an Ambarzumian-type theorem similar to what was done in \cite{BKAmba} in the case of quantum graphs, see also \cite{Davies:2013,BKS,KurasovSuhrAsymp}. This result falls into the realm of inverse spectral theory which seems to have been initiated by the work of Ambarzumian~\cite{Ambarz} and later Borg~\cite{Borg}. For an overview regarding Ambarzumian-type theorems for Schrödinger operators on quantum graphs we refer to \cite{KurasovBook} and references therein.
    \begin{cor}[Ambarzumian-type theorem for discrete graphs]\label{AmbarzumianGraphs}
         Let $L_0$ satisfy Assumption~\ref{ass:discrete-spec}, $c\colon X\to[0,\infty)$ and $L_c$ be the induced realization. Assume $\frac{c}{m} \in \ell^1(X,\mathbf{1})$. If $\lambda_n(c) = \lambda_n(0)$ for all $n=1,\dots,\dim \ell^2(X,m)$, then we have  $c = 0$.
    \end{cor}
	
	\subsection*{Acknowledgement}{P.B.~gratefully acknowledges support by the DFG
(Grant 397230547). C.R.~gratefully acknowledges support by the DFG (Grant 540199605) and thanks the FernUniversit\"at in Hagen for the hospitality during his visit.}
	
	
	\def\cprime{$'$} \def\polhk#1{\setbox0=\hbox{#1}{\ooalign{\hidewidth \lower1.5ex\hbox{`}\hidewidth\crcr\unhbox0}}} \def\cprime{$'$} \def\polhk#1{\setbox0=\hbox{#1}{\ooalign{\hidewidth \lower1.5ex\hbox{`}\hidewidth\crcr\unhbox0}}}

\end{document}